\documentclass[11pt]{article}

\usepackage{amsmath, amssymb, amsthm} 

\usepackage{dsfont}

\allowdisplaybreaks
\expandafter\let\expandafter\oldproof\csname\string\proof\endcsname
\let\oldendproof\endproof
\renewenvironment{proof}[1][\proofname]{%
	\oldproof[\bf #1]%
}{\oldendproof}

\allowdisplaybreaks

\theoremstyle{plain}
\newtheorem{theorem}{Theorem}
\newtheorem{lemma}{Lemma}[section]
\newtheorem{claim}[lemma]{Claim}
\newtheorem{proposition}[lemma]{Proposition}

\newtheorem*{claim*}{Main Claim}

\newcommand{\poly}{\text{poly}}
\RequirePackage[normalem]{ulem} 
\RequirePackage{color}\definecolor{RED}{rgb}{1,0,0}\definecolor{BLUE}{rgb}{0,0,1} 

\begin{document}

\title{A Characterization of Easily Testable Induced Digraphs and $k$-Colored Graphs}

\author{Lior Gishboliner\thanks{ETH Zurich. 
		Email: lior.gishboliner@math.ethz.ch.}
}

\maketitle

\begin{abstract}
	We complete the characterization of the digraphs $D$ for which the induced $D$-removal lemma has polynomial bounds, answering a question of Alon and Shapira. We also study the analogous problem for $k$-colored complete graphs. In particular, we prove a removal lemma with polynomial bounds for Gallai colorings. 
\end{abstract}

\section{Introduction}
In this paper we are concerned with binary combinatorial objects, such as graphs, digraphs and $k$-colored graphs. A removal lemma is a statement of the following form: Suppose that $G,F$ are binary combinatorial objects of the same type, where we think of $G$ as large and of $F$ as small and fixed. For every $\varepsilon > 0$ there is $\delta = \delta(\varepsilon) > 0$, such that if $G$ contains at most $\delta n^{v(F)}$ (induced) copies of $F$, then $G$ can be made (induced) $F$-free by changing at most $\varepsilon n^2$ entries in its adjacency matrix, where $n = v(G)$. 

The first result of this type was the famous {\em triangle removal lemma} of Ruzsa and Szemer\'edi \cite{RuzsaSz76}, which states that if an $n$-vertex graph contains at most $\delta(\varepsilon)n^3$ triangles, then it can be made triangle-free by deleting at most $\varepsilon n^2$ edges. This result played a key role in the development of extremal combinatorics, and its proof is one of the first applications of the celebrated Szemer\'edi regularity lemma \cite{Szemeredi}. The original proof generalizes from triangles to arbitrary graphs, giving the {\em graph removal lemma}. Later, Alon, Fischer, Krivelevich and Szegedy \cite{AFKS} proved an analogous result for induced subgraphs, the so-called {\em induced removal lemma}. This result states that if a graph contains at most $\delta(\varepsilon)n^{v(F)}$ induced copies of $F$, then it can be made induced $F$-free by adding/deleting at most $\varepsilon n^2$ edges. Analogous results have later been proved for other combinatorial structures, such as digraphs \cite{AS_digraphs} and ordered graphs \cite{ABF}. In another direction, the induced removal lemma was generalized to arbitrary hereditary graph properties \cite{AS_hereditary}. 

A common feature of all of the above results is that their proof uses Szemer\'edi's regularity lemma or a generalization thereof. Consequently, these proofs give quite weak, tower-type (or worse) bounds on $\delta(\varepsilon)$. For example, in the case of the graph removal lemma, the best known bound \cite{Fox} is 
$1/\delta \leq \text{tower}(O(\log 1/\varepsilon))$, where $\text{tower}(x)$ is a tower of $x$ exponents. For the induced removal lemma and for other structures (e.g. ordered graphs), the best known general bounds are even worse, see e.g. \cite{CF} for the state of the art. However, for particular graphs $F$, better bounds are known. This has raised the natural question of characterizing the cases where one can prove a removal lemma with polynomial bounds, i.e. when $1/\delta$ can be taken to be polynomial in $1/\varepsilon$.  
The first result of this type was obtained by Alon \cite{Alon}, who proved that for a graph $F$, the $F$-removal lemma has polynomial bounds if and only if $F$ is bipartite. Later, Alon and Shapira \cite{AS_induced} obtained a nearly complete characterization for the induced case, showing that the induced $F$-removal lemma has polynomial bounds if $F \in \{P_2,P_3,\overline{P_2},\overline{P_3}\}$, and that it does not have polynomial bounds if $F \notin \{P_2,P_3,\overline{P_2},\overline{P_3},P_4,C_4,\overline{C_4}\}$, where $P_k$ and $C_k$ are the path and cycle with $k$ vertices, respectively, and $\overline{F}$ denotes graph complement. The case of $P_4$ was later settled by Alon and Fox \cite{AF}. The author and Shapira \cite{GS_C4} proved an exponential bound for the case of $C_4$. For similar results for certain families of graph properties, see \cite{GS_poly}. 

Similar characterizations of polynomial removal lemmas were also obtained for other combinatorial structures, e.g. for tournaments \cite{FGSY} and for digraphs \cite{AS_digraphs}. In particular, Alon and Shapira \cite{AS_digraphs} characterized the digraphs $D$ for which the $D$-removal lemma has polynomial bounds, and asked for a characterization in the induced case. They showed \cite{AS_induced} that the induced $D$-removal lemma does not have polynomial bounds whenever $v(D) \geq 5$. Here we answer their question by completing the characterization. Before stating our results, let us introduce the following commonly used terminology: 
for a graph/digraph $F$, we say that (induced) $F$-freeness is {\em easily testable} if the (induced) $F$-removal lemma has polynomial bounds, and otherwise we say that it is {\em hard to test} (or just {\em hard}). This terminology comes from the field of property testing, where the goal is to design fast algorithms which distinguish between graphs satisfying a certain property and graphs which are $\varepsilon$-far from the property. The efficiency of such testers is measured by the number of queries they make to the input graph, and for many properties one can design testers whose query complexity is independent of the size of the input, i.e. depends only on $\varepsilon$. For hereditary graph properties, the query complexity of the best (one-sided error) tester is essentially given by the function $\delta(\varepsilon)$ in the corresponding removal lemma. We refer the reader to the book \cite{Goldreich} for an introduction to property testing.  


The following theorem gives the characterization of digraphs $D$ for which the induced $D$-removal lemma has polynomial bounds.

\begin{theorem}\label{thm:digraphs}
	For a digraph $D$, induced $D$-freeness is easily testable if and only if $v(D) = 2$. 
\end{theorem}
\noindent
Observe that the ``if'' part of Theorem \ref{thm:digraphs} is trivial.

Induced digraphs can encode $3$-colored complete graphs, where the color of a pair $\{i,j\}$ is the number of directed edges between $i$ and $j$, namely $0$, $1$ or $2$ (as in \cite{AS_digraphs}, we allow anti-parallel edges, but not parallel edges).  
By {\em $k$-colored complete graph} we mean a coloring of the edges of a complete graph with $k$ colors. So in particular, a graph can be thought of as a $2$-colored complete graph.
The removal lemma generalizes in a straightforward manner to $k$-colored complete graphs (where instead of edge addition/deletion, one speaks of edge color changes).

Now, given a digraph $D$, let $C(D)$ denote the corresponding $3$-colored complete graph; namely, $C(D)$ has the same vertex-set as $D$, and the color of a pair $\{i,j\}$ is the number of edges in $D$ between $i$ and $j$ (either $0$, $1$ or $2$). Note that the map $D \mapsto C(D)$ is not one-to-one. Indeed, if $D'$ is obtained from $D$ by reversing the direction of some single edges (i.e. edges $(i,j) \in E(D)$ for which $(j,i) \notin E(D)$), then $C(D') = C(D)$. 

Two subgraphs of a graph/digraph/$k$-colored graph are called {\em pair-disjoint} if they share at most one vertex. Throughout the paper, we will use the obvious fact that if a graph/digraph/$k$-colored graph $G$ contains $\varepsilon n^2$ pair-disjoint (induced) copies of $F$, then one must add/delete/change the color of at least $\varepsilon n^2$ edges in order to make $G$ (induced) $F$-free. By a 
{\em hardness construction} for (induced) $F$-freeness, we mean a graph $G$ which contains a collection of $\varepsilon n^2$ (induced) pair-disjoint copies of $F$, but only $\delta n^{v(F)}$ (induced) copies of $F$ overall, where $\delta \ll \poly(\varepsilon)$ (namely, $\delta$ goes to $0$ faster than any polynomial in $\varepsilon$).
So a hardness construction (for every $\varepsilon$ and $n$) shows that (induced) $F$-freeness is hard to test.  
The following (almost immediate) proposition shows that for a digraph $D$, a hardness construction for $C(D)$-freeness implies a hardness construction for induced $D$-freeness.
\begin{proposition}\label{prop:coloring}
	Let $D$ be a digraph. For $\varepsilon,\delta > 0$ and $n \geq 1$, suppose that there is a $3$-colored complete graph $G$ on $n$ vertices which contains $\varepsilon n^2$ pair-disjoint copies of $C(D)$, but only $\delta n^{v(D)}$ copies of $C(D)$ overall. Then there is a digraph $G'$ on $n$ vertices which contains $\varepsilon n^2$ induced pair-disjoint copies of $D$, but only $\delta n^{v(D)}$ induced copies of $D$ overall. 
\end{proposition}
	
\begin{proof}
	Choose $G'$ such that $C(G') = G$, and such that each of the $\varepsilon n^2$ pair-disjoint copies of $C(D)$ in $G$ makes an induced copy of $D$ in $G'$. 
\end{proof}	

Proposition \ref{prop:coloring} suggests the problem of characterizing the easily testable $3$-colored complete graphs. It turns out that here the situation is somewhat different from that of induced digraphs: while all induced digraphs on at least $3$ vertices are hard (by Theorem \ref{thm:digraphs}), there is a $3$-colored complete graph on $3$ vertices which is easily testable, namely the rainbow triangle. 
This assertion is the main part of our next result, which characterizes the easily testable $3$-colored complete graphs:
\begin{theorem}\label{thm:3-colored_graphs}
	Let $F$ be a $3$-colored complete graph. Then $F$-freeness is easily testable if and only if $v(F) = 2$ or $F$ is the rainbow triangle. 
\end{theorem}
	
The main part in the proof of Theorem \ref{thm:3-colored_graphs} is to show that the property of having no rainbow triangles is easily testable. This is done in Section \ref{sec:3-colored}. The structure of $3$-colored complete graphs with no rainbow triangles (also called Gallai colorings) was described by a fundamental result of Gallai \cite{Gallai} (see also \cite{GS}). This result states that if $G$ has no rainbow triangles, then $G$ is obtained from a 2-colored complete graph by replacing each vertex with a 3-colored complete graph without rainbow triangles (and replacing edges with complete bipartite graphs of the same color). Moreover, every 3-colored complete graph obtained in this way has no rainbow triangles. 
This structure result (stated below as Lemma \ref{lem:Gallai}) will play a key role in the proof. 

There are two digraphs $D$ for which $C(D)$ is the rainbow triangle. Let us denote them by $D_1,D_2$. Even though the rainbow triangle is easily testable, it turns out that induced $D_i$-freeness is hard to test for each $i = 1,2$. 
Theorem \ref{thm:3-colored_graphs} does imply however that the property of avoiding {\em both} $D_1,D_2$ as induced subdigraphs is easily testable. These digraphs $D_1,D_2$ are the only cases of Theorem \ref{thm:digraphs} which are not covered by using Theorem \ref{thm:3-colored_graphs} and Proposition \ref{prop:coloring}.  

To complement Theorem \ref{thm:3-colored_graphs}, we show that for $k \geq 4$, there are no non-trivial easily testable $k$-colored complete graphs.
\begin{proposition}\label{prop:k-colored}
	Let $k \geq 4$ and let $F$ be a $k$-colored complete graph. Then $F$-freeness is easily testable if and only if $v(F) = 2$.
\end{proposition}
\noindent
The proof of the hardness direction of Theorems \ref{thm:digraphs}-\ref{thm:3-colored_graphs} and Proposition \ref{prop:k-colored} appears in Section \ref{sec:hard}. 

\section{Testing for Gallai colorings}\label{sec:3-colored}

In this section we prove that the property of having no rainbow triangles is easily testable. We restate this result as follows. 
\begin{theorem}\label{thm:Gallai_coloring}
	Let $\varepsilon > 0$ be small enough, and let $G$ be an $n$-vertex $3$-colored complete graph with at most $\varepsilon^{36} n^3$ rainbow triangles. Then $G$ can be made rainbow-triangle-free by changing the color of at most $\varepsilon n^2$ edges. 
\end{theorem}
The proof is similar in spirit to the argument used by Alon and Fox \cite{AF} to show that the property of being a cograph (or, equivalently, of having no induced path on four vertices) is easily testable. 
We now introduce the necessary definitions.
Let $\mathcal{P} = (V_1,\dots,V_m)$ be a vertex-partition of a $3$-colored complete graph. For colors $a,b \in [3]$, we say that $\mathcal{P}$ is {\em $(a,b)$-monochromatic} if each of the bipartite graphs $(V_i,V_j)$ is monochromatic in color $a$ or in color $b$. Denote by $E(\mathcal{P})$ the set of all edges which go between the sets $V_1,\dots,V_m$, and put $e(\mathcal{P}) := |E(\mathcal{P})| = \sum_{1 \leq i < j \leq m}{|V_i||V_j|}$. We say that $\mathcal{P}$ is {\em $\varepsilon$-close to being $(a,b)$-monochromatic} if one can turn $\mathcal{P}$ into an $(a,b)$-monochromatic partition by changing the color of at most $\varepsilon \cdot e(\mathcal{P})$ of the edges in $E(\mathcal{P})$. 
Gallai \cite{Gallai,GS} proved the following fundamental fact about colorings with no rainbow triangles. 
\begin{lemma}[\cite{Gallai,GS}]\label{lem:Gallai}
	If $G$ is a $3$-colored complete graph with $|V(G)| \geq 2$ and with no rainbow triangles, then there exist two colors $a,b \in [3]$ such that $G$ admits an $(a,b)$-monochromatic partition (with at least two parts). 
	Conversely, if $\mathcal{P}$ is an $(a,b)$-monochromatic partition of $G$ (for some two colors $a,b$), and $G[X]$ has no rainbow triangles for every $X \in \mathcal{P}$, then $G$ has no rainbow triangles. 
\end{lemma}

\noindent
Before proceeding, let us prove the following very simple lemma:
\begin{lemma}\label{lem:balanced}
	Let $m,d,a_1,\dots,a_p \geq 0$ such that $a_1 + \dots + a_p = m$ and $a_i \leq m - d$ for every $1 \leq i \leq p$. Then $\sum_{1 \leq i < j \leq p}{a_ia_j} > d \cdot \frac{m-d}{2}$. 
\end{lemma}	
\begin{proof}
	Without loss of generality, assume that $a_1 \leq \dots \leq a_p$. Let $1 \leq i \leq p$ be minimal with $a_1 + \dots + a_i \geq d$. We have $i \leq p-1$, because otherwise we would have $a_p > m - d$, a contradiction. 
	Note that $a_{i+1} + \dots + a_p \geq a_{i+1} \geq a_i$ and $a_{i+1} + \dots + a_p = m - (a_1 + \dots + a_i) = m - (a_1 + \dots + a_{i-1}) - a_i > 
	m - d - a_i$. Summing these two inequalities and dividing by $2$, we obtain that $a_{i+1} + \dots + a_p > \frac{m-d}{2}$. Now,
	$
	\sum_{1 \leq i < j \leq p}{a_ia_j} \geq (a_1 + \dots + a_i) \cdot (a_{i+1} + \dots + a_p) > d \cdot \frac{m - d}{2},
	$
	as required. 
\end{proof}

	The main step in the proof of Theorem \ref{thm:Gallai_coloring} is the following approximate version of Lemma \ref{lem:Gallai}. It states that if a $3$-colored complete graph $G$ has few rainbow triangles, then for some two colors $a,b$, $G$ has a partition which is close to being $(a,b)$-monochromatic. 

	\begin{lemma}\label{lem:approximate partition}
		Let $\varepsilon > 0$ be small enough, and let $G$ be an $n$-vertex $3$-colored complete graph with at most $\varepsilon^{33} n^3$ rainbow triangles. Then there exist two colors $a,b \in [3]$ and a partition $\mathcal{P}$ of $V(G)$ which is $\varepsilon$-close to being $(a,b)$-monochromatic. 
	\end{lemma}
	\begin{proof}
		We will denote the color of an edge $\{x,y\}$ by $c(x,y)$. 
		Let $d_i(x)$ be the degree of $x$ in color $i$ (for $i \in [3]$).
		If there is a vertex $x$ and a color $i \in [3]$ such that $d_i(x) \geq (1-\varepsilon)(n-1)$, then the partition $\{x\},V(G) \setminus \{x\}$ satisfies the requirement in the lemma. So from now on, suppose that $d_i(x) < (1-\varepsilon)(n-1)$ for every vertex $x$ and color $i \in [3]$. 

		If there are less than $\varepsilon \binom{n}{2}$ edges of some color $i \in [3]$, then we can take $\mathcal{P}$ to be the partition of $V(G)$ into singletons (taking $a,b$ to be the two colors which are not $i$). So we may assume that for each color $i \in [3]$, there are at least $\varepsilon \binom{n}{2}$ edges of color $i$.
		 
		For a color $i \in [3]$, let $\text{SMALL}_i$ be the set of all vertices $v$ with $d_i(v) \leq \frac{\varepsilon^2}{128}n$. 
		Note that $|\text{SMALL}_i| \leq (1-\frac{\varepsilon}{2})n$ because otherwise there would be less than $\varepsilon \binom{n}{2}$ edges in color $i$. 
		Set
		$$
		s := \frac{128\log(2000/\varepsilon^2)}{\varepsilon^2},
		$$ 
		$$
		\delta := \frac{\varepsilon^2}{64s^2},
		$$
		$$
		k := \frac{128}{\varepsilon^2},
		$$
		and
		$$
		t := \frac{2(k+s\log s)}{\delta}.
		$$
		Note that $s = \tilde{O}(\frac{1}{\varepsilon^2})$, $\delta = \tilde{\Omega}(\varepsilon^6)$ and $t = \tilde{O}(\frac{1}{\varepsilon^8})$. 
		We will later use the fact that
		\begin{equation}\label{eq:s,t}
		2^k e^{-\delta t} \ll s^{-s},
		\end{equation}
		which easily follows from our choice of $t$. Here, the $\ll$ means that the left-hand side is smaller than $C$ times the right hand side for a fixed constant $C$, provided that $\varepsilon$ is small enough. 
		
		Sample $s + kt$ vertices of $G$ uniformly at random and independently. Let $S$ be the set of the first $s$ vertices, $T_1$ be the set of the next $t$ vertices, $T_2$ the set of the next $t$, and so on. 
		Put $T := T_1 \cup \dots \cup T_k$ and $R := S \cup T$. 
		Let $E_0$ be the event that $G[R]$ contains no rainbow triangles. 
		We have $|R| = s + kt = \tilde{O}(\frac{1}{\varepsilon^{10}}) \leq \frac{1}{2\varepsilon^{11}}$, say, where the last inequality holds if $\varepsilon$ is small enough.  
		Since $G$ contains at most $\Delta := \varepsilon^{33} n^3$ rainbow triangles, the probability that $G[R]$ contains a rainbow triangle is at most 
		$\binom{|R|}{3} \cdot \Delta \cdot 6 \cdot \frac{1}{n^3} \leq 
		\frac{|R|^3\Delta}{n^3} \leq \frac{1}{8}$. Namely, $\mathbb{P}[E_0] \geq \frac{7}{8}$. 
		
		Say that a vertex $v \in V(G)$ is {\em bad} if there is a color $i \in [3]$ such that $v \notin \text{SMALL}_i$, and yet $v$ has no neighbour of color $i$ in $S$. Observe that if $v \in V(G) \setminus \text{SMALL}_i$, then the probability that $v$ has no color-$i$ neighbour in $S$ is at most 
		$(1 - \varepsilon^2/128)^s \leq e^{-s\varepsilon^2/128} = \frac{\varepsilon^2}{2000}$. 
		Hence, the expected number of bad vertices is at most $3 \cdot \frac{\varepsilon^2}{2000} \cdot n \leq \frac{\varepsilon^2}{512}n$. Let $E_1$ be the event that there are at most $\frac{\varepsilon^2}{128}n$ bad vertices. By Markov's inequality, $\mathbb{P}[E_1] \geq \frac{3}{4}$. 
		
		Let $Z$ be the set of all vertices $v$ such that all edges between $v$ and $S$ have the same color. 
		For a vertex $v$, recall that $d_i(v) < (1 - \varepsilon)(n-1) \leq n-1 - \varepsilon n /2$ for every color $i$. 
  		Hence, \linebreak $\mathbb{P}[v \in Z] \leq 3 \cdot (1-\varepsilon/2)^s \leq 
  		3 e^{-\varepsilon s/2}$. It follows that $\mathbb{E}[|Z|] \leq 3 e^{-\varepsilon s/2}n$, and hence $\mathbb{P}[|Z| \geq 24 e^{-\varepsilon s/2}n] \leq \frac{1}{8}$ by Markov's inequality. 
  		Let $E_2$ be the event that $T \cap Z = \emptyset$. 
  		By the union bound, we have $\mathbb{P}[T \cap Z \neq \emptyset] \leq kt \cdot |Z|/n$. 
  		So if 
  		$|Z| \leq 24 e^{-\varepsilon s/2}n$, then 
  		$\mathbb{P}[T \cap Z \neq \emptyset] \leq kt \cdot 24 e^{-\varepsilon s/2} \leq \frac{1}{8}$, where the last inequality holds for $\varepsilon$ small enough, as $e^{\varepsilon s/2}$ is (at least) exponential in $1/\varepsilon$, while $kt$ is polynomial in $1/\varepsilon$. So overall, 
  		$\mathbb{P}[E_2] \geq \frac{3}{4}$. 
		
		For a partition $S = U_1 \cup \dots \cup U_p$ and for a set $X \supseteq S$, we say that a partition $X = W_1 \cup \dots \cup W_q$ (where $q \geq p$) {\em extends} $(U_1,\dots,U_p)$ if $W_i \cap S = U_i$ for every $1 \leq i \leq p$. 
		
		Suppose that $E_0$ and $E_2$ happened. 
		Since $E_0$ happened, by Lemma \ref{lem:Gallai} there exists an $(a,b)$-monochromatic partition $R = W_1 \cup \dots \cup W_q$, $q \geq 2$, for some two colors $a,b$. 
		Let $U_i := W_i \cap S$, and suppose without loss of generality that $U_1,\dots,U_{p}$ are the nonempty sets among $U_1,\dots,U_q$. We claim that 
		$p \geq 2$. Indeed, if $p = 1$ then $S \subseteq W_1$. But then, for any $w \in W_2$, we have that all edges between $w$ and $S$ have the same color. This however implies that $w \in Z$, contradicting that $E_2$ happened. 
		So we see that if $E_0$ and $E_2$ happened, then there exist two colors $a,b$, a partition 
		$S = U_1 \cup \dots \cup U_p$ with $p \geq 2$, and an $(a,b)$-monochromatic partition $R = W_1 \cup \dots \cup W_q$ which extends $(U_1,\dots,U_p)$.
		The main step in the proof is to establish the following:
		

		\begin{claim*}\label{statement:approximate_partition main}
			Fix any choice of $S$ and suppose that $E_1$ happened. Then, either there is a partition $\mathcal{P}$ as in the statement of the lemma, or with probability larger than $\frac{3}{4}$ the following holds: for every two colors $a,b \in [3]$ and for every $(a,b)$-monochromatic partition $S = U_1 \cup \dots \cup U_p$ with $p \geq 2$, there is no $(a,b)$-monochromatic partition of $R$ which extends $(U_1,\dots,U_p)$. 
		\end{claim*}
		\begin{proof}
			Fix two colors $(a,b)$ and an $(a,b)$-monochromatic partition $S = U_1 \cup \dots \cup U_p$ with $p \geq 2$. Without loss of generality, suppose that $a = 1, b = 2$. For $1 \leq i < j \leq p$, let $c_{i,j} \in \{1,2\}$ be the color of the monochromatic bipartite graph $(U_i,U_j)$.  
			Let $\mathcal{A}$ be the event that there is a $(1,2)$-monochromatic partition of $R$ which extends $(U_1,\dots,U_p)$. 
			We will show that either there is a partition $\mathcal{P}$ as in the statement of the lemma, or $\mathbb{P}[\mathcal{A}] \ll s^{-s}$.  
			We then take the union bound over all at most $3 \cdot s^s$ choices of $a,b$ and $(U_1,\dots,U_p)$, to get the required result.

			Let us define sets $V_i^{(\ell)}$, $1 \leq \ell \leq k$ and $1 \leq i \leq p$, as follows. The definition is by induction on $\ell$. 
			For $1 \leq i \leq p$, define $V^{(1)}_i$ as the set of vertices $x \in V(G) \setminus S$ such that there is an edge of color $3$ between $x$ and $U_i$.
			Since $E_1$ happened, 
			all but at most $\frac{\varepsilon^2}{128} n$ of the vertices in $V(G) \setminus \text{SMALL}_3$ belong to $V_1^{(1)} \cup \dots \cup V_p^{(1)}$. 
			Hence, $|V_1^{(1)} \cup \dots \cup V_p^{(1)}| \geq n - |\text{SMALL}_3| - \frac{\varepsilon^2}{128} n \geq
			\frac{\varepsilon}{2}n - \frac{\varepsilon^2}{128} n \geq \frac{\varepsilon}{4}n$. 
			Now let $2 \leq \ell \leq k$, and suppose we already defined $V^{(\ell-1)}_1,\dots,V^{(\ell-1)}_p$. For $1 \leq i \leq p$, let $V^{(\ell)}_i$ be the set of all $x \in V(G) \setminus S$ such that either $x \in V^{(\ell-1)}_i$, or 
			there are at least $\delta n$ edges of color $3$ between $x$ and $V^{(\ell-1)}_i$, or
			for each color $c \in \{1,2\}$, there are at least $\delta n$ edges of color $c$ between $x$ and $V^{(\ell-1)}_i$. 
			
			\begin{claim}\label{claim:forcing}
				Let $1 \leq \ell \leq k$, $1 \leq i \leq p$ and $x \in V^{(\ell)}_i$. Then with probability at least $1 - (2^{\ell} - 2)e^{-\delta t}$ over the choice of $T_1,\dots,T_{\ell-1}$, the following holds: 
				if $(W_1,\dots,W_q)$ is a $(1,2)$-monochromatic partition of $S \cup T_1 \cup \dots \cup T_{\ell-1} \cup \{x\}$ which extends $(U_1,\dots,U_p)$, then $x \in W_i$. 
			\end{claim}
			\begin{proof}
				We prove the statement by induction on $\ell$. For $\ell = 1$, it follows immediately from the definition of the set $V^{(1)}_i$ that if $(W_1,\dots,W_q)$ is a $(1,2)$-monochromatic partition of $S \cup \{x\}$ which extends $(U_1,\dots,U_p)$, then $x \in W_i$ (with probability $1$). Let now $\ell \geq 2$, and let $x \in V^{(\ell)}_i$. If $x \in V^{(\ell-1)}_i$, then the assertion follows from the induction hypothesis. 
				Otherwise, either there are at least $\delta n$ edges of color $3$ between $x$ and $V^{(\ell-1)}_i$, or for each color $c \in \{1,2\}$, there are at least $\delta n$ edges of color $c$ between $x$ and $V^{(\ell-1)}_i$. We will assume that the latter case holds; the former case can be handled similarly (and more easily).
				So for each $c = 1,2$, let $Y_c \subseteq V^{(\ell-1)}_i$ be a set of at least $\delta n$ vertices which are connected to $x$ in color $c$.
				The probability that $T_{\ell-1} \cap Y_1 = \emptyset$ or $T_{\ell-1} \cap Y_2 = \emptyset$ is at most $2(1 - \delta)^{t} \leq 2e^{-\delta t}$. Suppose that $T_{\ell-1} \cap Y_1 \neq \emptyset$ and $T_{\ell-1} \cap Y_2 \neq \emptyset$, and fix vertices $y_c \in T_{\ell-1} \cap Y_c$, $c = 1,2$. 
				By the induction hypothesis, with probability at least $1 - 2 \cdot (2^{\ell-1} - 2)e^{-\delta t} = 1 - (2^{\ell} - 4)e^{-\delta t}$ over the choice of $T_1,\dots,T_{\ell-2}$, the following holds: for every $(1,2)$-monochromatic partition $(W_1,\dots,W_q)$ of $S \cup T_1 \cup \dots \cup T_{\ell-2} \cup \{y_1,y_2\}$ which extends $(U_1,\dots,U_p)$, it holds that $y_1,y_2 \in W_i$. 
				Assume that this event holds. 
				Let $(W_1,\dots,W_q)$ be a $(1,2)$-monochromatic partition of $S \cup T_1 \cup \dots \cup T_{\ell-1} \cup \{x\}$ which extends $(U_1,\dots,U_p)$. We have that $y_1,y_2 \in W_i$ and $\{x,y_1\},\{x,y_2\}$ have different colors. Hence, $x$ must be in $W_i$ as well. The probability that this fails is at most 
				$2e^{-\delta t} + (2^{\ell} - 4)e^{-\delta t} = (2^{\ell} - 2)e^{-\delta t}$, as required. 
			\end{proof}
 			
 			\begin{claim}\label{claim:balanced}
 				Let $1 \leq \ell \leq k$ and $1 \leq i \leq p$ and suppose that $|V^{(\ell)}_i| \geq (1 - \frac{\varepsilon}{2})n$. Then $\mathbb{P}[\mathcal{A}] \ll s^{-s}$. 
 			\end{claim}
 			\begin{proof}
 				For convenience, let us assume that $i = 1$. Fix a vertex $u \in U_2$. By our assumption, there are at least $\varepsilon (n-1) - |S| \geq \frac{3\varepsilon}{4}n$ vertices $x \in V(G) \setminus S$ such that the color of $\{x,u\}$ is not $c_{1,2}$. 
 				Since $|V^{(\ell)}_1| \geq (1 - \frac{\varepsilon}{2})n$, at least $\frac{\varepsilon}{4}n$ of these vertices are in $V^{(\ell)}_1$. 
 				The probability that $T_k$ contains no such vertex $x \in V^{(\ell)}_1$ is at most $(1 - \frac{\varepsilon}{4})^t \leq e^{-\varepsilon t/4}$. Suppose that $T_k$ contains such a vertex $x$. By Claim \ref{claim:forcing}, with probability at least $1 - (2^k - 2) e^{-\delta t}$ over the choice of $T_1,\dots,T_{k-1}$, it holds that if $(W_1,\dots,W_q)$ is a $(1,2)$-monochromatic partition of $S \cup T_1 \cup \dots \cup T_{k-1} \cup \{x\}$ extending $(U_1,\dots,U_p)$, then $x \in W_1$. Assume that this event happens; we show that then $\mathcal{A}$ fails. Indeed, suppose by contradiction that $(W_1,\dots,W_q)$ is a $(1,2)$-monochromatic partition of $R = S \cup T_1 \cup \dots \cup T_k$ extending $(U_1,\dots,U_p)$. We have $x \in T_k$, so $x \in W_1$. However, the color of $\{x,u\}$ is not $c_{1,2}$, contradicting the fact that the bipartite graph $(W_1,W_2)$ is monochromatic with color $c_{1,2}$. It follows that 
 				$
 				\mathbb{P}[\mathcal{A}] \leq e^{-\varepsilon t/4} + (2^k - 2) e^{-\delta t} \leq 2^k e^{-\delta t} \ll s^{-s}$, where the last inequality holds by \eqref{eq:s,t}. This proves the claim. 
 			\end{proof}
 		
 			\begin{claim}\label{claim:approximate_monochromatic}
 				Let $1 \leq \ell \leq k$ and $1 \leq i < j \leq p$, and suppose that there are at least $\delta n^2$ edges between $V^{(\ell)}_i$ and $V^{(\ell)}_j$ whose color is not $c_{i,j}$. Then $\mathbb{P}[\mathcal{A}] \ll s^{-s}$.
 			\end{claim}
 			\begin{proof}
 				The probability that $T_k$ contains no edge $\{v_i,v_j\} \in E(V^{(\ell)}_i,V^{(\ell)}_j)$ with $c(v_i,v_j) \neq c_{i,j}$, is at most $(1 - 2\delta)^{t/2} \leq e^{-\delta t}$. Suppose that $T_k$ contains such an edge $\{v_i,v_j\}$. By applying Claim \ref{claim:forcing} to $v_i$ and $v_j$, we get the following: with probability at least $1 - 2 \cdot (2^k - 2) e^{-\delta t}$ over the choice of $T_1,\dots,T_{k-1}$, it holds that if $(W_1,\dots,W_q)$ is a $(1,2)$-monochromatic partition of $S \cup T_1 \cup \dots \cup T_{k-1} \cup \{v_i,v_j\}$ extending $(U_1,\dots,U_p)$, then $v_i \in W_i$ and $v_j \in W_j$. But $c(v_i,v_j) \neq c_{i,j}$, contradicting the fact that the bipartite graph $(W_i,W_j)$ should be monochromatic in color $c_{i,j}$. The probability of failure is at most 
 				$e^{-\delta t} + 2 \cdot (2^k - 2) e^{-\delta t} \leq 2 \cdot 2^k e^{-\delta t} \ll s^{-s}$, where the last inequality holds by \eqref{eq:s,t}. 
 			\end{proof}
 			
 			Put $V^{(\ell)} := V^{(\ell)}_1 \cup \dots \cup V^{(\ell)}_p$. By construction, we have $V^{(\ell)} \subseteq V^{(\ell+1)}$ for every $\ell \geq 1$. By our choice of $k$, there must be some $1 \leq \ell \leq k-1$ such that $|V^{(\ell+1)}| \leq |V^{(\ell)}| + \frac{\varepsilon^2}{128}n$. From now on we fix such an $\ell$. 
 			By Claims \ref{claim:balanced} and \ref{claim:approximate_monochromatic}, we may assume that: 
 			\begin{enumerate}
 				\item[(a)] $|V^{(\ell)}_i| \leq (1 - \frac{\varepsilon}{2})n$ for every $1 \leq i \leq p$.
 				\item[(b)] For every pair $1 \leq i < j \leq p$,
 				all but at most $\delta n^2$ of the edges between $V^{(\ell)}_i$ and $V^{(\ell)}_j$ \nolinebreak have \nolinebreak color \nolinebreak $c_{i,j}$.
 			\end{enumerate}
 			We would like the sets $V^{(\ell)}_1,\dots,V^{(\ell)}_p$ to be pairwise disjoint; to this end, if an element belongs to several of these sets, then we place it in one of them arbitrarily, removing it from all others. Items (a)-(b) continue to hold. This also does not change $V^{(\ell)}$. 
 			
 			Recall that 
			$|V^{(\ell)}| \geq |V^{(1)}| \geq \frac{\varepsilon}{4}n$. 
 			Put $V' := V^{(\ell+1)} \setminus V^{(\ell)}$
 			and note that $|V'| \leq \frac{\varepsilon^2}{128}n$ by our choice of $\ell$. Also, set $X := V(G) \setminus (S \cup V^{(\ell)} \cup V') = V(G) \setminus (S \cup V^{(\ell+1)})$. Observe that by definition, if 
			$x \in X$ 
 			then for every $1 \leq i \leq p$, all but at most $2\delta n$ of the edges between $x$ and $V^{(\ell)}_i$ have the same color $c_{x,i} \in \{1,2\}$ (because $x \notin V^{(\ell+1)}$). 
 			Consider the partition $\mathcal{P}$ of $V(G)$ having the following parts: $V^{(\ell)}_1,\dots,V^{(\ell)}_p$; $V' \cup S$; and the vertices of 
 			$X$ as singletons.  
			We claim that $e(\mathcal{P}) \geq \frac{\varepsilon}{16}n^2$. Indeed, first note that $|V^{(\ell)}| + |X| = n - |V' \cup S| \geq (1 - \frac{\varepsilon}{8})n$, say. 
			If 
			$|X| \geq \frac{\varepsilon}{8}n$ 
			then $e(\mathcal{P}) \geq |V^{(\ell)}| \cdot |X| \geq 
			(n - \frac{\varepsilon}{8}n - |X|) \cdot |X| \geq 
			(n - \frac{\varepsilon}{4}n) \cdot \frac{\varepsilon}{8}n \geq \frac{\varepsilon}{16}n^2$.
			On the other hand, if 
			$|X| \leq \frac{\varepsilon}{8}n$ then 
			$|V^{(\ell)}| \geq (1 - \frac{\varepsilon}{4})n$. Now, as $|V^{(\ell)}_i| \leq (1 - \frac{\varepsilon}{2})n$ for every $1 \leq i \leq p$, we get from Lemma \ref{lem:balanced} (with parameters $m = (1 - \frac{\varepsilon}{4})n$ and $d = \frac{\varepsilon}{4}n$) that 
			$e(\mathcal{P}) \geq 
			\sum_{1 \leq i < j \leq p}{|V^{(\ell)}_i| \cdot |V^{(\ell)}_j|} \geq 
			\frac{\varepsilon}{4}n \cdot \frac{(1 - \frac{\varepsilon}{2})n}{2} \geq
			\frac{\varepsilon}{16}n^2$.
			So indeed $e(\mathcal{P}) \geq\frac{\varepsilon}{16}n^2$ in both cases. 
 			
 			We now modify at most $\varepsilon \cdot e(\mathcal{P})$ of the edges in $E(\mathcal{P})$ to turn $\mathcal{P}$ into a $(1,2)$-monochromatic partition.
 			The changes we make are as follows:
 			\begin{itemize}
 				\item For every $1 \leq i < j \leq p$, make $(V^{(\ell)}_i,V^{(\ell)}_j)$ monochromatic in color $c_{i,j}$. This is a total of at most $\binom{p}{2} \cdot \delta n^2 \leq s^2 \delta n^2 = \frac{\varepsilon^2}{64}n^2$ edge changes altogether.
 				\item For each $x \in X$ and $1 \leq i \leq p$, make all edges between $x$ and $V^{(\ell)}_i$ have the same color $c_{x,i} \in \{1,2\}$; this can be done with at most $2\delta n$ edge changes. Thus, this step requires at most $n \cdot p \cdot 2\delta n \leq 2s\delta n^2 \leq \frac{\varepsilon^2}{64}n^2$ edge changes altogether. 
 				\item Change the color of all color-$3$ edges inside $X$. Recall that all but at most $\frac{\varepsilon^2}{128} n$ of the vertices in $X$ are in $\text{SMALL}_3$, because $X \cap V^{(1)} = \emptyset$ and as $E_1$ happened.
 				Recall also that each vertex in $\text{SMALL}_3$ is incident to at most $\frac{\varepsilon^2}{128} n$ edges of color $3$. Hence, this step requires at most $\frac{\varepsilon^2}{128} n \cdot |X| + |X| \cdot \frac{\varepsilon^2}{128} n \leq \frac{\varepsilon^2}{64} n^2$ edge changes.
 				\item Color all edges between $V' \cup S$ and $V(G) \setminus (V' \cup S)$ with color $1$ (say). This step requires at most $|V' \cup S| \cdot n \leq \frac{\varepsilon^2}{64}n^2$ edge changes. 
 			\end{itemize}
 			The total number of edge changes in the above four items is at most $\frac{\varepsilon^2}{16} n^2 \leq \varepsilon \cdot e(\mathcal{P})$. 
 			After these changes, $\mathcal{P}$ is $(1,2)$-monochromatic.
 			This proves the main claim. 
		\end{proof}
		Let us now complete the proof of Lemma \ref{lem:approximate partition} using the main claim. Suppose by contradiction that there is no partition $\mathcal{P}$ of $V(G)$ as in the statement of the lemma. Then by the main claim, and as $\mathbb{P}[E_1] \geq 3/4$, we have the following: with probability larger than $1/2$, there does not exist an $(a,b)$-monochromatic partition $S = U_1 \cup \dots \cup U_p$, where $p \geq 2$, and an $(a,b)$-monochromatic partition of $R$ which extends $(U_1,\dots,U_p)$. On the other hand, we saw that such partitions $S = U_1 \cup \dots \cup U_p$ and $R = W_1 \cup \dots \cup W_q$ do exist if $E_0$ and $E_2$ happen, which has probability at least $1/2$ as $\mathbb{P}[E_0],\mathbb{P}[E_2] \geq 3/4$. This contradiction completes the proof. 
	\end{proof}

	\begin{proof}[Proof of Theorem \ref{thm:Gallai_coloring}]
		We decompose $G$ by repeatedly applying Lemma \ref{lem:approximate partition}. It is convenient to describe the decomposition using a tree, where each node corresponds to a subset of $V(G)$. The root is $V(G)$. At each step, if there is a leaf $X$ with $|X| \geq \varepsilon n$, then apply Lemma \ref{lem:approximate partition} to $G[X]$. As $\varepsilon^{33}|X|^3 \geq \varepsilon^{36}n^3$, we know that $G[X]$ contains at most $\varepsilon^{33}|X|^3$ rainbow triangles. Thus, Lemma \ref{lem:approximate partition} gives a partition $\mathcal{P}_X$ of $X$ which is $\varepsilon$-close to being $(a,b)$-monochromatic, for some two colors $a,b \in [3]$. Now add all the sets $Y \in \mathcal{P}_X$ as children of $X$. When this process stops, every leaf is of size at most $\varepsilon n$. For each non-leaf $X$, turn $\mathcal{P}_X$ into an $(a,b)$-monochromatic partition (for the two suitable colors $a,b$) by changing the colors of at most $\varepsilon \cdot e(\mathcal{P})$ of the edges in $E(\mathcal{P})$. This requires in total at most $\varepsilon\binom{n}{2} \leq \varepsilon n^2/2$ edge changes altogether. Next, for each leaf $X$, make $G[X]$ rainbow-triangle-free. This requires at most 
		$$
		\sum_{X \text{ leaf}}{\binom{|X|}{2}} \leq \frac{\varepsilon n - 1}{2} \cdot
		\sum_{X \text{ leaf}}{|X|} \leq \varepsilon n^2/2
		$$ 
		additional edge changes. So the total number of edge-changes is at most $\varepsilon n^2$. After these edge-changes, the resulting $3$-colored complete graph has no rainbow triangles, by the ``conversely'' part of Lemma \ref{lem:Gallai}. This completes the proof. 
	\end{proof}

\section{Lower Bounds}\label{sec:hard}
In this section we prove the ``only if'' parts of Theorems \ref{thm:digraphs} and \ref{thm:3-colored_graphs} and of Proposition \ref{prop:k-colored}. 
The proofs use Behrend-type constructions, similarly to \cite{Alon,AS_induced}. Due to this similarity, we will be somewhat brief.  
We need the following simple lemma. 
%
%

\begin{lemma}\label{lem:design}
	For $d \geq 2$ and $r \geq 2d$, there is a collection $R \subseteq [r]^d$, $|R| \geq  (r/2)^2$, such that any two $d$-tuples in $R$ agree on at most one coordinate. 
\end{lemma}

\begin{proof}
	Let $p$ be a prime such that $r/2< p\leq r$; such a prime exists by Bertrand's postulate. For $a,b \in \mathbb{F}_p$, let $x_{a,b} \in \mathbb{F}_p^d$ be the $d$-tuple $x_{a,b}(i) = a + (i-1)b$, $i = 1,\dots,d$. 
	Observe that for $(a_1,b_1) \neq (a_2,b_2)$, there is at most one $1 \leq i \leq d$ with $x_{a_1,b_1}(i) = x_{a_2,b_2}(i)$. Indeed, if there are two such $1 \leq i < j \leq d$, then $a_1 + (i-1)b_1 = a_2 + (i-1)b_2$ and $a_1 + (j-1)b_1 = a_2 + (j-1)b_2$. Solving this system of equations gives $a_1 = a_2$ and $b_1 = b_2$, a contradiction. Here we use the fact that $i \not\equiv j \pmod{p}$, which follows from $p > r/2 \geq d$. 
\end{proof}

\begin{lemma}\label{lem:RS_triangle}
	Let $k \geq 2$, let $F$ be a $k$-colored complete graph, and suppose that there is a triangle in $F$ whose edges avoid one of the colors. Then for every small enough $\varepsilon > 0$ and large enough $n$, there is an $n$-vertex $k$-colored complete graph $G$ which contains $\varepsilon n^2$ pair-disjoint copies of $F$, but only $\varepsilon^{\Omega(\log1/\varepsilon)}n^{v(F)}$ copies of $F$ altogether. 
\end{lemma}
\begin{proof}
	Put $f = v(F)$ and suppose that $V(F) = [f]$. Without loss of generality, suppose that $F$ has a triangle whose edges avoid the color $k$. 
	By \cite[Lemma 4.1]{AS_induced}, for every $m \geq 1$, there is a set $S \subseteq [m]$ of size at least $m/e^{C\sqrt{\log m}}$, such that for all 
	$1 \leq p,q \leq f-1$, there is no solution to $p s_1 + q s_2 = (p+q) s_3$ with distinct $s_1,s_2,s_3 \in S$.
	Let $m$ be the maximal integer satisfying $e^{-C\sqrt{\log m}} \geq 4f^4\varepsilon$. It is easy to check that $m \geq (1/\varepsilon)^{\Omega(\log 1/\varepsilon)}$. Let $S \subseteq [m]$ be as above; so $|S| \geq 4f^4 \varepsilon m$.  
	Define a $k$-colored complete graph $H$ consisting of $f$ disjoint sets $V_1,\dots,V_f$, where $V_i = [i \cdot m]$. So $v(H) = \binom{f+1}{2}m$, and hence $f^2m/2 \leq v(H) \leq f^2m$. 
	For each $x \in [m]$ and $s \in S$, add a copy $F_{x,s}$ of $F$ in which $v_i := x + (i-1)s \in V_i$ plays the role of $i$ for every $i \in [f] = V(F)$. All edges in $H$ which do not belong to any of the copies $F_{x,s}$ (in particular, the edges inside the sets $V_1,\dots,V_f$) are colored with color $k$.  
	We claim that the copies $F_{x,s}$, $(x,s) \in [m] \times S$, are pair-disjoint.  
	Indeed, if $F_{x_1,s_1}$ and $F_{x_2,s_2}$ have the same vertex in $V_i$ and $V_j$, then $x_1 + (i-1)s_1 = x_2 + (i-1)s_2$ and $x_1 + (j-1)s_1 = x_2 + (j-1)s_2$. Solving this system of equations, we get that $x_1 = x_2$ and $s_1 = s_2$. So we conclude that the copies $F_{x,s}$ are indeed pair-disjoint. The number of these copies is $m|S| \geq 4f^4 \varepsilon m^2$.
	
	Next, we bound the number of triangles in $H$ which avoid the color $k$. 
	Such a triangle cannot contain two vertices from the same $V_i$, since the edges inside $V_1,\dots,V_f$ are colored with color $k$. Let $1 \leq a < b < c \leq f$, and let $x \in V_a, y \in V_b, z \in V_c$ be a triangle avoiding the color $k$. By construction, there are $s_1,s_2,s_3 \in S$ such that 
	$y - x = (b - a)s_1$, $z - y = (c - b)s_2$ and $z - x = (c - a)s_3$. So, setting $p := b - a$ and $q := c - b$, we have $ps_1 + qs_2 = (p + q)s_3$. By our choice of $S$, we have $s_1 = s_2 = s_3 =: s$. So each such triangle is determined by the choice of $x \in V_a$ and $s$. There are $|V_a| = a \cdot m \leq f \cdot m$ choices for $x$ and $|S| \leq m$ choices for $s$. Hence, the total number of triangles
	in $H$ avoiding the color $k$ is at most $\binom{f}{3} \cdot fm \cdot |S| \leq f^4m^2$.  
	
	Now let $G$ be the $\frac{n}{v(H)}$-blowup of $H$, where all edges inside the blowup of each $V_i$ are colored with $k$. Each copy of $F$ in $H$ gives rise to $(\frac{n}{2v(H)})^2$ pair-disjoint copies of $F$ in $G$, by Lemma \ref{lem:design} with parameters $r=\frac{n}{v(H)}$ and $d = f$. Hence, $G$ contains a collection of 
	$4f^4\varepsilon m^2 \cdot (\frac{n}{2v(H)})^2 \geq \varepsilon n^2$ pair-disjoint copies of $F$. To complete the proof, we bound the total number of copies of $F$ in $G$. Each copy of $F$ must contain a triangle which avoids the color $k$. Each triangle avoiding the color $k$ in $G$ must come from a triangle avoiding the color $k$ in $H$. The number of such triangles in $H$ is at most $f^4m^2$, and each of these triangles in $H$ gives rise to $(\frac{n}{v(H)})^3$ such triangles in $G$. Hence, the total number of triangles in $G$ avoiding the color $k$ is at most $f^4m^2 \cdot (\frac{n}{v(H)})^3 \leq f^4m^2 \cdot (\frac{2n}{f^2m})^3 \leq \frac{n^3}{m}$. It follows that the number of copies of $F$ in $G$ is at most $\frac{n^3}{m} \cdot n^{f-3} = \frac{n^f}{m} \leq \varepsilon^{\Omega(\log 1/\varepsilon)} \cdot n^f$, as required. 
\end{proof}
Lemma \ref{lem:RS_triangle} immediately implies that for $k \geq 4$, every $k$-colored complete graph with at least $3$ vertices is hard. For $k = 3$, observe that if $F$ is a $3$-colored complete graph and two of the edges incident to some $v \in V(F)$ have the same color, then $F$ has a triangle avoiding one of the colors. If $v(F) \geq 5$ then every vertex $v$ is like that (since there are at least 4 edges incident to $v$ and only 3 colors). If $v(F) = 3$ then such a vertex exists unless $F$ is a rainbow triangle. And if $v(F) = 4$ then such a vertex exists unless each color spans a matching of size $2$. Let $F_4$ denote this $3$-colored complete graph; namely, $V(F_4) = \{a_1,a_2,a_3,a_4\}$; $\{a_1,a_2\},\{a_3,a_4\}$ have color $1$; $\{a_1,a_4\},\{a_2,a_3\}$ have color $2$; and $\{a_1,a_3\},\{a_2,a_4\}$ have color $3$.
To complete the proof of Theorem \ref{thm:3-colored_graphs}, we now describe a variant of the above construction suited for $F_4$. 
\begin{lemma}\label{lem:RS_F4}
	For every small enough $\varepsilon > 0$ and large enough $n$, there is an $n$-vertex $3$-colored complete graph $G$ which contains $\varepsilon n^2$ pair-disjoint copies of $F_4$, but only $\varepsilon^{\Omega(\log1/\varepsilon)}n^{4}$ copies of $F_4$ altogether.
\end{lemma}
\begin{proof}
	By \cite[Lemma 3.1]{Alon}, for every $m \geq 1$, there is a set $S \subseteq [m]$ of size at least $m/e^{C\sqrt{\log m}}$ containing no solution to $s_1 + s_2 + s_3 = 3s_4$ with distinct $s_1,s_2,s_3,s_4$.
	Let $m$ be the maximal integer satisfying $e^{-C\sqrt{\log m}} \geq 400\varepsilon$. It is easy to check that $m \geq (1/\varepsilon)^{\Omega(\log 1/\varepsilon)}$. Let $S \subseteq [m]$ be as above; so $|S| \geq 400\varepsilon m$.  
	Define a $3$-colored complete graph $H$ consisting of $4$ disjoint sets $V_1,V_2,V_3,V_4$, where $V_i = [i \cdot m]$; so $v(H) = 10m$. 
	For each $x \in [m]$ and $s \in S$, add a copy $F_{x,s}$ of $F_4$ on the vertices $v_i = x + (i-1)s \in V_i$, where $v_i$ plays the role of $a_i$ for each $1 \leq i \leq 4$. 
	All edges not participating in one of these copies are colored with color $3$. 
	Observe that all edges between $V_1$ and $V_3$ and between $V_2$ and $V_4$ have color $3$. 
	As before, the copies $F_{x,s}$ are pair-disjoint. Their number is $m|S| \geq 400 \varepsilon m^2$.
	
	Observe that if $F$ is a copy of $F_4$ in $H$, then $F$ must contain one vertex from each of the sets $V_1,\dots,V_4$. 
	Indeed,	note that for every pair $1 \leq i < j \leq 4$, the edges in $V_i \cup V_j$ use only two colors. So $|V(F) \cap (V_i \cup V_j)| \leq 2$ for all $i,j$ (since every triangle in $F_4$ is rainbow). Hence, $|V(F) \cap V_i| = 1$ for every $1 \leq i \leq 4$. It is now easy to see that every copy of $F_4$ in $H$ is of the form $v_1,\dots,v_4$, where $v_i \in V_i$ plays the role of $a_i$. Fix such a copy $v_1,\dots,v_4$. By construction, there must be $s_1,s_2,s_3,s_4$ such that 
	$v_2 - v_1 = s_1$, $v_3 - v_2 = s_2$, $v_4 - v_3 = s_3$ and $v_4 - v_1 = 3s_4$. So $s_1 + s_2 + s_3 = 3s_4$, and hence $s_1 = s_2 = s_3 = s_4$ by our choice of $S$. It follows that the number of copies of $F_4$ in $H$ is $m|S| \leq m^2$.
	
	Let $G$ be the $\frac{n}{v(H)}$-blowup of $H$, where all edges inside the blowup of each $V_i$ are colored with color $3$. Each copy of $F_4$ in $H$ gives rise to $(\frac{n}{2v(H)})^2$ pair-disjoint copies of $F_4$ in $G$ by Lemma \ref{lem:design}. Hence, $G$ contains a collection of 
	$400 \varepsilon m^2 \cdot (\frac{n}{2v(H)})^2 = \varepsilon n^2$ pair-disjoint copies of $F_4$. 
	Let us now upper-bound the total number of copies of $F_4$ in $G$. By the same argument as above, every copy of $F_4$ in $G$ must be of the form $v_1,\dots,v_4$ with $v_i$ belonging to the blowup of $V_i$ and playing the role of $a_i$ in the copy. So
	every copy of $F_4$ in $G$ corresponds to a copy of $F_4$ in $H$. On the other hand, every copy of $F_4$ in $H$ gives rise to $(\frac{n}{v(H)})^4$ copies of $F_4$ in $G$. So overall, there are at most $m^2 \cdot (\frac{n}{v(H)})^4 \leq \frac{n^4}{m} \leq \varepsilon^{\Omega(\log1/\varepsilon)}n^{4}$ copies of $F_4$ in $G$, as required.
\end{proof}

To complete the proof of Theorem \ref{thm:digraphs}, we need to handle the two digraphs $D$ whose corresponding $3$-colored complete graph $C(D)$ is the rainbow triangle. These digraphs are obtained from each other by reversing the direction of all edges. So by symmetry, it remains to handle just one of them. Let then $D_3$ be the digraph with vertices $a_1,a_2,a_3$ and edges $(a_1,a_3),(a_2,a_3),(a_3,a_2)$. 
\begin{lemma}
	For every small enough $\varepsilon > 0$ and large enough $n$, there is an $n$-vertex digraph $G$ which contains $\varepsilon n^2$ pair-disjoint induced copies of $D_3$, but only $\varepsilon^{\Omega(\log1/\varepsilon)}n^3$ induced copies of $D_3$ altogether.
\end{lemma}
\begin{proof}
	By \cite[Lemma 3.1]{Alon}, for every $m \geq 1$, there is a set $S \subseteq [m]$ of size at least $m/e^{C\sqrt{\log m}}$ containing no solution to $s_1 + s_2 = 2s_3$ with distinct $s_1,s_2,s_3$.
	Let $m$ be the maximal integer satisfying $e^{-C\sqrt{\log m}} \geq 144\varepsilon$. It is easy to check that $m \geq (1/\varepsilon)^{\Omega(\log 1/\varepsilon)}$. Let $S \subseteq [m]$ be as above; so $|S| \geq 144\varepsilon m$.  
	Define a digraph $H$ consisting of $3$ disjoint sets $V_1,V_2,V_3$, where $V_i = [i \cdot m]$; so $v(H) = 6m$. 
	For each $x \in [m]$ and $s \in S$, add a copy $D_{x,s}$ of $D_3$ on the vertices $v_1 = x \in V_1$, $v_2 = x + s \in V_2$, $v_3 = x + 2s \in V_3$, where $v_i$ plays the role of $a_i$ for each $1 \leq i \leq 3$. 
	For all pairs of vertices $\{x,y\}$ not participating in one of these copies, put exactly one edge between $x$ and $y$, and if $x \in V_1, y \in V_3$ then direct this edge from $y$ to $x$. This way, the only edges going from $V_1$ to $V_3$ are those participating in one of the copies $D_{x,s}$. Note that, in particular, each of the sets $V_1,V_2,V_3$ spans a tournament. 
	As before, the copies $D_{x,s}$ are pair-disjoint. Their number is $m|S| \geq 144 \varepsilon m^2$.
	
	It is easy to check that every induced copy of $D_3$ in $H$ must be of the form $v_1,v_2,v_3$ with $v_i \in V_i$ playing the role of $a_i$. If $v_1,v_2,v_3$ is such a copy, then by construction there are $s_1,s_2,s_3 \in S$ with $v_2 - v_1 = s_1$, $v_3 - v_2 = s_2$ and $v_3 - v_1 = 2s_3$. So $s_1 + s_2 = 2s_3$, implying that $s_1 = s_2 = s_3$. It follows that $H$ contains at most $|S|m \leq m^2$ induced copies of $D_3$.
	
	Let $G$ be the $\frac{n}{v(H)}$-blowup of $H$, where the blowup of each $V_i$ is a tournament. Every induced copy of $D_3$ in $H$ gives rise to $(\frac{n}{2v(H)})^2$ pair-disjoint induced copies of $D_3$ in $G$, by Lemma \ref{lem:design}. Hence, $G$ contains a collection of $144\varepsilon m^2 \cdot (\frac{n}{2v(H)})^2 = \varepsilon n^2$ pair-disjoint induced copies of $D_3$. On the other hand, it is easy to see that every induced copy of $D_3$ in $G$ corresponds to an induced copy of $D_3$ in $H$, so overall $G$ has at most $m^2 \cdot (\frac{n}{v(H)})^3 \leq \frac{n^3}{m} \leq \varepsilon^{\Omega(\log1/\varepsilon)}n^3$ induced copies of $D_3$. 
\end{proof}

\end{document}